\documentclass[10pt,draft]{article}
\usepackage{amsmath,amsthm,amssymb,amsfonts}
\usepackage{enumerate}
\usepackage{mathrsfs}
\usepackage[cmtip,all]{xy}

\numberwithin{equation}{section}
\numberwithin{figure}{section}

\normalsize

\newtheoremstyle{theoremstyle}
  {10pt}      
  {5pt}       
  {\itshape}  
  {}          
  {\bfseries} 
  {:}         
  {.5em}      
  {}          

\newtheoremstyle{examplestyle}
  {10pt}      
  {5pt}       
  {}          
  {}          
  {\bfseries} 
  {:}         
  {.5em}      
  {}          

\theoremstyle{theoremstyle}
\newtheorem{theorem}{Theorem}[section]
\newtheorem*{theorem*}{Theorem}
\newtheorem{lemma}[theorem]{Lemma}
\newtheorem{proposition}[theorem]{Proposition}
\newtheorem*{proposition*}{Proposition}
\newtheorem{corollary}[theorem]{Corollary}
\newtheorem*{corollary*}{Corollary}

\newtheorem{definition*}{Definition}
\newtheorem{remark}[theorem]{Remark}
\newtheorem{remark*}{Remark}

\newcommand{\G}{\mathbf{G}_{\mathfrak{m}}}

\newcommand{\KU}{\mathbf{KU}}
\newcommand{\KO}{\mathbf{KO}}
\newcommand{\ku}{\mathbf{ku}}
\newcommand{\ko}{\mathbf{ko}}
\newcommand{\KGL}{\mathbf{KGL}}
\newcommand{\KQ}{\mathbf{KQ}}
\newcommand{\KW}{\mathbf{KW}}
\newcommand{\MGL}{\mathbf{MGL}}

\newcommand{\kgl}{\mathbf{kgl}}
\newcommand{\kq}{\mathbf{kq}}
\newcommand{\E}{\mathbf{E}}
\newcommand{\e}{\mathbf{e}}
\newcommand{\FF}{\mathbf{F}}
\newcommand{\f}{\mathsf{f}}
\newcommand{\s}{\mathsf{s}}
\newcommand{\ii}{\mathsf{i}}
\newcommand{\rr}{\mathsf{r}}

\newcommand{\scc}{\mathsf{sc}}
\newcommand{\LL}{\mathbf{L}}
\newcommand{\vcd}{\mathsf{vcd}}
\newcommand{\MSS}{\mathbf{MSS}}

\newcommand{\holim}{\mathrm{holim}}

\newcommand{\Char}{\mathsf{char}}

\newcommand{\MZ}{\mathbf{M}\mathbf{Z}}

\newcommand{\unit}{\mathbf{1}}
\newcommand{\Z}{\mathbb{Z}}

\newcommand{\Sm}{\mathbf{Sm}}
\newcommand{\SH}{\mathbf{SH}}

\newcommand{\HH}{\mathbf{H}}

\newcommand{\HZ}{\mathbf{H}\mathbf{Z}}
\newcommand{\HV}{\mathbf{H}\mathbf{V}}

\newcommand{\map}{\mathrm{map}}

\newcommand{\eff}{\mathsf{eff}}

\newcommand{\colim}{\mathrm{colim}}
\newcommand{\op}{\mathsf{op}}

\newcommand{\hocolim}{\mathrm{hocolim}}

\title{{\bf The motivic Hopf map solves the homotopy limit problem for $K$-theory}}
\author{Oliver R\"ondigs, Markus Spitzweck, Paul Arne {\O}stv{\ae}r}
\date{January 20, 2017}
\begin{document}
\maketitle
\begin{abstract}
We solve affirmatively the homotopy limit problem for $K$-theory over fields of finite virtual cohomological dimension.
Our solution employs the motivic slice filtration and the first motivic Hopf map.
\end{abstract}

\section{Introduction}
\label{section:introduction}

A homotopy limit problem asks for an equivalence between fixed points and homotopy fixed points for a group action \cite{Thomason}.
In some contexts, 
the fixed points are easily described, 
and one then obtains a description of the otherwise intractable homotopy fixed points.
Many distinguished results in algebraic topology take the form of a homotopy limit problem, 
e.g., 
the Atiyah-Segal completion theorem linking equivariant $K$-theory to representation theory, 
Segal's Burnside ring conjecture on stable cohomotopy, 
Sullivan's conjecture on the homotopy type of real points of algebraic varieties, 
and the Quillen-Lichtenbaum conjecture on Galois descent for algebraic $K$-theory under field extensions.

In this paper we give a surprising solution of the homotopy limit problem for $K$-theory in the stable motivic homotopy category.
This is achieved by analyzing the slice filtration for algebraic and hermitian $K$-theory \cite{R-O}, \cite{roendigs-spitzweck-oestvaer.slices-sphere}, \cite{voevodsky.open}, 
and completing with respect to the Hopf element $\eta$ in the Milnor-Witt $K$-theory ring \cite{Morel:ICM}.

To provide context for our approach, 
recall that complex conjugation of vector bundles gives rise to the Adams operation $\Psi^{-1}$ and an action of the group $C_{2}$ of order two on the complex $K$-theory spectrum $\KU$.
Atiyah \cite{Atiyah} shows there is an isomorphism
\begin{equation}
\label{equation:KOhlp}
\KO
\overset{\simeq}{\longrightarrow}
\KU^{hC_{2}}
\end{equation}
between the real $K$-theory spectrum $\KO$ and the $C_{2}$-homotopy fixed points of $\KU$. 
For the corresponding connective $K$-theory spectra, 
the homotopy cofiber of 
\begin{equation}
\label{equation:kohlp}
\ko
\longrightarrow
\ku^{hC_{2}}
\end{equation}
is an infinite sum $\bigvee_{i<0}\Sigma^{4i}\HH\Z/2$ of suspensions of the mod-$2$ Eilenberg-MacLane spectrum.

We are interested in the analogues of \eqref{equation:KOhlp} and \eqref{equation:kohlp} for algebraic $K$-theory $\KGL$ with $C_{2}$-action given by the Adams operation $\Psi^{-1}$ and hermitian $K$-theory $\KQ$,
see for example \cite[\S3,4]{R-O}.
Throughout we work in the stable motivic homotopy category $\SH$ over a field $F$ of characteristic $\Char(F)\neq 2$.
Our starting point is, 
somewhat unexpectedly in view of \eqref{equation:kohlp}, 
the naturally induced map between fixed points and homotopy fixed points 
\begin{equation}
\label{equation:kohlp2}
\gamma
\colon
\kq
\longrightarrow
\kgl^{hC_{2}}
\end{equation}
for the projections of $\KGL$ and $\KQ$ to the effective stable motivic homotopy category $\SH^{\eff}$.
The latter is the localizing subcategory of $\SH$ generated by suspension spectra of smooth schemes \cite{voevodsky.open}.
Let $\eta$ be the first motivic Hopf map induced by the natural map of algebraic varieties $\mathbf{A}^{2}\smallsetminus\{0\}\longrightarrow\mathbf{P}^{1}$.
Recall that $\eta$ defines a non-nilpotent element in the homotopy group $\pi_{1,1}\unit$ of the motivic sphere spectrum.
We let $\pi_{\star}\E$ denote the bigraded coefficients of a generic motivic spectrum $\E$.
The stable cone of $\eta$ acquires a Bousfield localization functor $\LL_{\eta}$ defined on all motivic spectra \cite[Appendix A]{Rondigs-Ostvar3}.
Let $\vcd_{2}(F)$ denote the mod-$2$ cohomological dimension of the absolute Galois group of $F(\sqrt{-1})$ \cite[Chapter 1,\S3]{serreGC}.
We solve the homotopy limit problem \eqref{equation:kohlp2} affirmatively by completing with respect to $\eta$.

\begin{theorem}
\label{theorem:kqetahomotopylimit}
Suppose $F$ is a field of $\Char(F)\neq 2$ and virtual cohomological dimension $\vcd_{2}(F)<\infty$.
Then \eqref{equation:kohlp2} induces an isomorphism
\begin{equation}
\label{equation:kqhlpeta}
\LL_{\eta}(\gamma)
\colon
{\kq^{\wedge}_{\eta}
\overset{\simeq}{\longrightarrow}
{\kgl^{hC_{2}}}}.
\end{equation}
\end{theorem}
We show that the homotopy fixed point spectrum $\kgl^{hC_{2}}$ in \eqref{equation:kqhlpeta} is $\eta$-complete. 
The proof of Theorem \ref{theorem:kqetahomotopylimit} invokes the slice filtration 
\begin{equation}
\label{equation:slicefiltration}
\cdots
\subset
\Sigma^{q+1}_{T}\SH^{\eff}
\subset
\Sigma^{q}_{T}\SH^{\eff}
\subset
\Sigma^{q-1}_{T}\SH^{\eff}
\subset
\cdots
\end{equation}
introduced by Voevodsky \cite[\S2]{voevodsky.open}. 
Using \eqref{equation:slicefiltration} one associates to $\E$ an integrally graded family of slices $\s_{\ast}(\E)$ and a trigraded slice spectral sequence 
\begin{equation}
\label{equation:sssintro}
\pi_{\star}\s_{\ast}(\E)\Longrightarrow\pi_{\star}\E.
\end{equation}
We show that \eqref{equation:sssintro} converges conditionally for the $\eta$-completion of $\kq$ and also for the homotopy fixed point spectrum $\kgl^{hC_{2}}$.
In contrast to the topological situation \eqref{equation:KOhlp}, 
$\pi_{\star}\KQ$ and $\pi_{\star}\KGL$ are unknown over general fields.
Nonetheless we obtain a proof of Theorem \ref{theorem:kqetahomotopylimit} by using the computations of $\s_{\ast}(\KQ)$ and $\s_{\ast}(\KGL^{hC_{2}})$ accomplished in \cite{R-O}.

Next we turn to solving the homotopy limit problem 
\begin{equation}
\label{equation:bigKQhlp}
\Upsilon
\colon
\KQ
\longrightarrow
\KGL^{hC_{2}}.
\end{equation} 
Here $\KGL^{hC_{2}}$ is $\eta$-complete essentially due to motivic orientability of algebraic $K$-theory.
We proceed by comparing with the effective cocovers of $\KQ$ and $\KGL^{hC_{2}}$.
Remarkably, 
the first motivic Hopf map $\eta$ turns $\Upsilon$ into an isomorphism without altering the target in \eqref{equation:bigKQhlp}.
\begin{theorem}
\label{theorem:KQetahomotopylimit}
Suppose $F$ is a field of $\Char(F)\neq 2$ and virtual cohomological dimension $\vcd_{2}(F)<\infty$.
Then \eqref{equation:bigKQhlp} induces an isomorphism
\begin{equation}
\label{equation:effKQhlpeta}
\LL_{\eta}(\Upsilon)
\colon
{\KQ^{\wedge}_{\eta}}
\overset{\simeq}{\longrightarrow}
{\KGL^{hC_{2}}}.
\end{equation}
\end{theorem}

Supplementing our main results we note that the $\eta$-arithmetic square 
\begin{equation}
\label{equation:etaarithmetic}
\xymatrix{
\KQ \ar[r] \ar[d] & \KW  \ar[d] \\
{\KQ^{\wedge}_{\eta}} \ar[r] & {\KQ^{\wedge}_{\eta}}[\eta^{-1}]
}
\end{equation}
for $\KQ$ \cite[\S3.1]{roendigs-spitzweck-oestvaer.slices-sphere} coincides up to isomorphism with the Tate diagram \cite[(20)]{HKO}
\begin{equation}
\label{equation:tatearithmetic}
\xymatrix{
\KQ \ar[r] \ar[d] & \KW \ar[d] \\
\KGL^{hC_{2}} \ar[r] & {\KGL}^{tC_{2}}
}
\end{equation}
for the $C_{2}$-action on $\KGL$. 
Here $\KW$ denotes the higher Witt-theory and ${\KGL}^{tC_{2}}$ denotes the Tate $K$-theory spectrum.
Moreover,
by representability, 
\eqref{equation:effKQhlpeta} implies that for every $X\in\Sm_{F}$ --- smooth $F$-schemes of finite type --- there is a naturally induced isomorphism
\begin{equation*}
\LL_{\eta}(\Upsilon)_{\star}
\colon
\pi_{\star}
{\KQ(X)^{\wedge}_{\eta}}
\overset{\cong}{\longrightarrow}
\pi_{\star}{\KGL(X)^{hC_{2}}}.
\end{equation*}

\begin{remark}
The earlier works \cite{BKSO} and \cite{HKO} identified the $2$-adic completion of the homotopy fixed points by showing an isomorphism 
$\pi_{\star}\KQ/2\overset{\cong}{\longrightarrow}\pi_{\star}\KGL^{hC_{2}}/2$.
Explicit calculations are carried out over the complex numbers $\mathbb{C}$ in \cite{IS}.
However, 
for the identification of ${\KGL^{hC_{2}}}$ in \eqref{equation:effKQhlpeta} it is paramount to work in $\SH$, 
so that the $\eta$-completion of $\KQ$ makes sense.
The commonplace assumption $\vcd_{2}(F)<\infty$ is also used in \cite{BKSO}, \cite{HKO}, 
and in the context of the Quillen-Lichtenbaum conjecture for \'etale $K$-theory \cite[\S4]{MR2544389}.
\end{remark}

\section{The first motivic Hopf map $\eta$}
\label{section:tmhm}

We view $\mathbf{A}^{2}\smallsetminus\{0\}$ and $\mathbf{P}^{1}$ as motivic spaces pointed at $(1,1)$ and $[1:1]$, 
respectively.
The canonical projection map $\mathbf{A}^{2}\smallsetminus\{0\}\longrightarrow\mathbf{P}^{1}$ induces the stable motivic Hopf map 
$\eta\colon\G\longrightarrow\unit$ for the motivic sphere spectrum $\unit$.
Iteration of $\eta$ yields the cofiber sequence
\begin{equation}
\label{equation:modetapowers}
\G^{\wedge n}
\overset{\eta^{n}}{\longrightarrow} 
\unit
\longrightarrow 
\unit/\eta^{n}.
\end{equation}
The $\eta$-completion $\E^{\wedge}_{\eta}$ of a motivic spectrum $\E$ is defined as the homotopy limit $\underset{n\rightarrow\infty}{\holim}\;\E/\eta^{n}$ 
of the canonically induced diagram
\begin{equation}
\label{equation:modetapowerstower}
\dots
\longrightarrow 
\E/\eta^{n+1}
\longrightarrow 
\E/\eta^{n}
\longrightarrow
\dots
\longrightarrow 
\E/\eta.
\end{equation}
By \eqref{equation:modetapowers} and \eqref{equation:modetapowerstower} there is a naturally induced map
\begin{equation}
\label{equation:etacompletionmap}
\E
\longrightarrow 
\E^{\wedge}_{\eta}.
\end{equation}
We say that $\E$ is {\it $\eta$-complete} if the map in \eqref{equation:etacompletionmap} is an isomorphism.
The Bousfield localization $\LL_{\eta}\E$ of $\E$ for the cone of $\eta$ coincides with $\E^{\wedge}_{\eta}$.
Recall that the algebraic cobordism spectrum $\MGL$ is the universal oriented motivic spectrum, 
see \cite{NSOfieldsinstitute}, \cite{PPRhha}.
\begin{lemma}
\label{lemma:orientedetaequivalence}
Every module over an oriented motivic ring spectrum is $\eta$-complete. 
\end{lemma}
\begin{proof}
The unit map for algebraic cobordism $\unit\longrightarrow\MGL$ factors through the cone $\unit/\eta$, 
see \cite[Lemma 3.24]{roendigs-spitzweck-oestvaer.slices-sphere} for an explicit factorization, 
which implies $\MGL\wedge\eta=0$.
The statement for modules follows readily.
\end{proof}

\section{The slice filtration}
\label{section:sf}
In this section we discuss results for the slice filtration \cite{voevodsky.open} which will be applied in the proofs of our main results in Section \ref{section:proofsofmainresults}.  
Throughout we work over a base field $F$.

To \eqref{equation:slicefiltration} one associates distinguished triangles 
\begin{equation}
\label{equation:slicedefinition}
\f_{q+1}(\E)
\longrightarrow 
\f_{q}(\E)
\longrightarrow 
\s_{q}(\E), 
\end{equation}
for every motivic spectrum $\E$, 
see \cite[Theorem 2.2]{voevodsky.open}. 
Here the $q$th effective cover of $\E$ is the universal map $\f_{q}(\E)\longrightarrow\E$ from $\Sigma^{q}_{T}\SH^{\eff}$ to $\E$.
The $q$th slice $\s_{q}(\E)\in\Sigma^{q}_{T}\SH^{\eff}$ is uniquely determined up to isomorphism by \eqref{equation:slicedefinition}.
Every object of $\Sigma^{q+1}_{T}\SH^{\eff}$ maps trivially to $\s_{q}(\E)$.
It is technically important for many constructions to have a ``strict model'' for the slice filtration, 
e.g., 
by means of model categories as in \cite[\S3.1]{GRSO}, \cite[\S3.2]{Pelaez}.
We note that $\f_{q}\s_{q'}\simeq\s_{q'}\f_{q}$ follows from \cite[(2.2), \S6]{GRSO} for all $q,q'\in\Z$.

\begin{lemma}
\label{lemma:slicefiltrationcomplete}
The slice filtration is exhaustive in the sense that there is an isomorphism
\begin{equation}
\label{equation:exhaustive}
\underset{q\rightarrow-\infty}{\hocolim}\;\f_{q}(\E)
\overset{\simeq}{\longrightarrow}
\E.
\end{equation}
\end{lemma}
\begin{proof}
Each generator $\Sigma^{s,t}X_{+}$ of $\SH$ is contained in $\Sigma^{q'}_{T}\SH^{\eff}$ for some $q'\in\Z$.
Here $s,t\in\Z$ and $X\in\Sm_{F}$, 
see for example \cite[Theorem 9.1]{DD}.
Recall that $\f_{q'}$ preserves homotopy colimits \cite[Corollary 4.5]{Spitzweck}, \cite[Lemma 4.2]{voevodsky.open}.
By the universal property of the $q'$th effective cover it suffices to show there is an isomorphism
\begin{equation*}
\SH(\Sigma^{s,t}X_{+},\underset{q<q'}{\hocolim}\;\f_{q'}\f_{q}(\E))
\overset{\simeq}{\longrightarrow}
\SH(\Sigma^{s,t}X_{+},\f_{q'}(\E)).
\end{equation*}
This follows since $\f_{q'}\f_{q}\simeq\f_{q'}$ for $q<q'$.
\end{proof}

\begin{lemma}
\label{lemma:slicesareetacomplete}
The slices of a motivic spectrum are $\eta$-complete. 
\end{lemma}
\begin{proof}
Every slice $\s_{q}(\E)$ is a module over the motivic ring spectrum $\s_{0}(\unit)$,  
cf.~\cite[\S6 (iv),(v)]{GRSO} and \cite[Theorem 3.6.13(6)]{Pelaez}. 
If $F$ is a perfect field, 
then $\s_{0}(\unit)$ is the motivic cohomology spectrum $\MZ$ by \cite[Theorem 10.5.1]{Levine:slices} and \cite[Theorem 6.6]{Voevodsky:zero-slice}.
This extends by base change;
every field is essentially smooth over a perfect field \cite[Lemma 2.9]{HKPAO}, 
and \cite[Lemma 2.7(1)]{HKPAO} verifies the hypothesis of \cite[Theorem 2.12]{pp:functoriality} for an essentially smooth map.
To conclude the proof we use Lemma \ref{lemma:orientedetaequivalence} and the canonical orientation on $\MZ$ \cite[\S10]{NSOdm}. 
\end{proof}

\begin{corollary}
\label{corollary:Ktheoryetacomplete}
Algebraic $K$-theory $\KGL$ and its effective cover $\kgl$ are $\eta$-complete.
\end{corollary}
\begin{proof}
This follows from Lemma \ref{lemma:orientedetaequivalence} by using the orientation map $\MGL\longrightarrow\KGL$, 
see for example \cite[Example 2.4]{PPR} and \cite[Examples 2.1, 2.2]{SObott}, 
and the geometric fact that the algebraic cobordism spectrum is effective \cite[Corollary 3.2]{Spitzweck}, \cite[\S8]{voevodsky.open}.
\end{proof}
\begin{corollary}
\label{corollary:homotopyfixedpointsKtheoryetacomplete}
The homotopy fixed points spectra $\KGL^{hC_{2}}$ and $\kgl^{hC_{2}}$ are $\eta$-complete.
\end{corollary}
\begin{proof}
Let $\E$ be short for $\KGL$ or $\kgl$.
We use homotopy limits to model the homotopy fixed points $\E^{hC_{2}}$ for the $C_{2}$-action given by the Adams operation $\Psi^{-1}$ \cite[\S18]{Hirschhorn}, \cite[\S3,4]{R-O}.
Corollary \ref{corollary:Ktheoryetacomplete} implies there is an isomorphism
\begin{equation*}
\underset{C_{2}}{\holim}\;\E
\overset{\simeq}{\longrightarrow}
\underset{C_{2}}{\holim}\;
\underset{n\rightarrow\infty}{\holim}\;\E/\eta^{n}.
\end{equation*}
Thus the corollary follows by commuting homotopy limits over small categories, 
i.e., 
\begin{equation*}
\underset{C_{2}}{\holim}\;
\underset{n\rightarrow\infty}{\holim}\;\E/\eta^{n}
\overset{\simeq}{\longrightarrow}
\underset{n\rightarrow\infty}{\holim}\;
\underset{C_{2}}{\holim}\;\E/\eta^{n}.
\end{equation*}
\end{proof}

The $q$th effective cocover $\f^{q-1}(\E)$ of $\E$ is uniquely determined up to isomorphism by the distinguished triangle
\begin{equation}
\label{equation:effectivecocoverdefinition}
\f_{q}(\E)
\longrightarrow 
\E
\longrightarrow 
\f^{q-1}(\E).
\end{equation}
Note that $\f^{q-1}(\E)$ is a $(q-1)$-coeffective motivic spectrum, 
i.e., 
it is an object of the right orthogonal subcategory of $\Sigma^{q}_{T}\SH^{\eff}$.
If $q\leq q'$ the isomorphism 
\begin{equation}
\label{equation:effectiveiso}
\s_{q'}\f_{q}(\E)\overset{\simeq}{\longrightarrow}\s_{q'}(\E)
\end{equation}
implies $\s_{q'}\f^{q-1}(\E)\simeq\ast$.
When $q=0$, 
\eqref{equation:effectivecocoverdefinition} yields a distinguished triangle for the effective cover $\e\in\SH^{\eff}$ of $\E$, 
i.e., 
\begin{equation}
\label{equation:effectivecoverdefinition}
\e
\longrightarrow 
\E
\longrightarrow 
\f^{-1}(\E).
\end{equation}
We note that all the nonnegative slices of the coeffective motivic spectrum $\f^{-1}(\E)$ are trivial.

\begin{lemma}
\label{lemma:nonegativeslices}
If $\E$ has no nontrivial negative slices then $\E\in\SH^{\eff}$.
\end{lemma}
\begin{proof}
Using \eqref{equation:slicedefinition}, \eqref{equation:exhaustive} and \eqref{equation:effectivecoverdefinition} it follows that $\f^{-1}(\E)\simeq\ast$.
\end{proof}

\begin{lemma}
\label{lemma:sliceiso}
Suppose $\E\rightarrow\FF$ induces an isomorphism $\s_{q}\f^{q-1}(\E)\overset{\simeq}{\rightarrow}\s_{q}\f^{q-1}(\FF)$ for all $q\in\Z$.
Then there is a naturally induced isomorphism $\f^{q-1}(\E)\overset{\simeq}{\rightarrow}\f^{q-1}(\FF)$.
\end{lemma}
\begin{proof}
This follows by applying \eqref{equation:slicedefinition} and \eqref{equation:exhaustive} to the effective cocovers.
\end{proof}

\begin{lemma}
\label{lemma:negativedt}
For $n>0$ there is a distinguished triangle
\begin{equation}
\label{equation:negativedt}
\f_{-n+1}\f^{-1}(\E)
\longrightarrow
\f_{-n}\f^{-1}(\E)
\longrightarrow
\s_{-n}(\E).
\end{equation}
It follows that $\f_{-n}\f^{-1}(\E)$ is a finite extension of the negative slices of $\E$.
\end{lemma}
\begin{proof}
This follows from the distinguished triangles:
\begin{equation}
\label{equation:3by3diagram}
\xymatrix{
\f_{-n+1}(\e) \ar[r] \ar[d] &
\f_{-n+1}(\E) \ar[r] \ar[d] &
\f_{-n+1}\f^{-1}(\E) \ar[d] \\
\f_{-n}(\e) \ar[r] \ar[d] &
\f_{-n}(\E) \ar[r] \ar[d] &
\f_{-n}\f^{-1}(\E) \ar[d] \\
\s_{-n}(\e) \ar[r] &
\s_{-n}(\E) \ar[r] &
\s_{-n}\f^{-1}(\E) 
}
\end{equation}
In \eqref{equation:3by3diagram} the slice $\s_{-n}(\e)\simeq\ast$ by the assumption $n>0$.
This implies \eqref{equation:negativedt}.
We note the effective cover $\f_{0}\f^{-1}(\E)\simeq\ast$ since the map $\f_{0}(\e)\longrightarrow\f_{0}(\E)$ is an isomorphism.
It follows that $\f_{-1}\f^{-1}(\E)$ is isomorphic to $\s_{-1}(\E)$. 
The conclusion follows from \eqref{equation:negativedt} by induction on $n$.
\end{proof}

\begin{remark}
For $n>0$, 
$\f_{-n}\f^{-1}(\E)$ is $\eta$-complete by Lemmas \ref{lemma:slicesareetacomplete} and \ref{lemma:negativedt}.
\end{remark}

\begin{lemma}
\label{lemma:negativesliceisomorphism}
For $n\geq -q>0$ there are isomorphisms 
\begin{equation}
\label{equation:negativesliceisomorphism}
\s_{q}(\E)
\overset{\simeq}{\longrightarrow}
\s_{q}\f^{-1}(\E)
\overset{\simeq}{\longleftarrow}
\s_{q}\f_{-n}\f^{-1}(\E).
\end{equation}
\end{lemma}
\begin{proof}
Here we use that $\s_{q}(\e)\simeq\ast$ for $q<0$. 
The isomorphism for $\f_{-n}$ is a special case of \eqref{equation:effectiveiso}.
\end{proof}

For every $\E\in\SH$ the distinguished triangle \eqref{equation:effectivecocoverdefinition} yields a commutative diagram:
\begin{equation}
\label{equation:effectivediagram1}
\xymatrix{
\f_{q+1}(\E) \ar[r] \ar[d] &
\E \ar[r] \ar@{=}[d] &
\f^{q}(\E) \ar[d] \\ 
\f_{q}(\E) \ar[r] &
\E \ar[r] &
\f^{q-1}(\E) &
}
\end{equation}
The slice completion of $\E$ is defined as the homotopy limit 
\begin{equation}
\label{equation:slicecompletion}
\scc(\E)
\equiv 
\underset{q\rightarrow\infty}{\holim}\;\f^{q-1}(\E).
\end{equation}
Using \eqref{equation:effectivediagram1} and \eqref{equation:slicecompletion} we conclude there is a distinguished triangle
\begin{equation}
\label{equation:slicecompletedefinition}
\underset{q\rightarrow\infty}{\holim}\;\f_{q}(\E)
\longrightarrow 
\E
\longrightarrow 
\scc(\E).
\end{equation}
We say that $\E$ is {\it slice complete} if the homotopy limit $\underset{q\rightarrow\infty}{\holim}\;\f_{q}(\E)$ in \eqref{equation:slicecompletedefinition} is contractible. 

\begin{lemma}
\label{lemma:slicesareslicecomplete}
For every $\E\in\SH$, both $\s_{q}(\E)$ and $\f^{q}(\E)$ are slice complete for all $q\in\Z$.
\end{lemma}
\begin{proof}
If $q<q'$ there are distinguished triangles 
\begin{equation}
\label{equation:qqslicedefinition}
\f_{q+1}\f_{q'}(\E)
\overset{\simeq}{\longrightarrow}
\f_{q}\f_{q'}(\E)
\longrightarrow 
\s_{q}\f_{q'}(\E), \;
\f_{q'}\f_{q+1}(\E)
\overset{\simeq}{\longrightarrow}
\f_{q'}(\E)
\longrightarrow 
\f_{q'}\f^{q}(\E).
\end{equation}
It follows that $\f_{q'}\s_{q}(\E)\simeq\s_{q}\f_{q'}(\E)\simeq\ast$ and $\f_{q'}\f^{q}(\E)\simeq\ast$ for $q<q'$ by \eqref{equation:qqslicedefinition}.
\end{proof}

\begin{lemma}
\label{lemma:Ktheoryslicecomplete}
Algebraic $K$-theory $\KGL$ and its effective cover $\kgl$ are slice complete.
\end{lemma}
\begin{proof}
It suffices to consider $\KGL$.  
The associated Nisnevich sheaf of homotopy groups $\underline{\pi}_{p,q}\KGL$ is trivial when $p<2q$.
Hence $\f_{q}(\KGL)$ is $q$-connected by \cite[Lemma 3.17]{roendigs-spitzweck-oestvaer.slices-sphere}, 
i.e., 
for every triple $(s,t,d)$ of integers with $s-t+d<q$ and every $X\in\Sm_{F}$ of dimension $\leq d$, 
the group $[\Sigma^{s,t}X_{+},\f_{q}(\KGL)]$ is trivial. 
We conclude by letting $q\rightarrow \infty$.
\end{proof}

From Lemma \ref{lemma:Ktheoryslicecomplete} we deduce isomorphisms
\begin{equation*}
\KGL^{hC_{2}}
\overset{\simeq}{\longrightarrow}
\scc(\KGL)^{hC_{2}}, 
\kgl^{hC_{2}}
\overset{\simeq}{\longrightarrow}
\scc(\kgl)^{hC_{2}}.
\end{equation*}
However, 
it is unclear whether $\KGL^{hC_{2}}$ and $\kgl^{hC_{2}}$ are slice complete because homotopy fixed points need not commute with effective cocovers or equivalently with effective covers. 
To emphasize this issue we construct an example in \S\ref{section:appendix}, 
see also Proposition \ref{proposition:hfpslices}.

For $n>0$ there is a naturally induced distinguished triangle
\begin{equation}
\label{equation:longeffectivecocoverdefinition}
\e
\longrightarrow 
\f_{-n}(\E)
\longrightarrow 
\f_{-n}\f^{-1}(\E).
\end{equation}
Lemma \ref{lemma:slicefiltrationcomplete}, \eqref{equation:effectivecoverdefinition}, and \eqref{equation:longeffectivecocoverdefinition} imply there is a naturally induced isomorphism
\begin{equation}
\label{equation:positivemap}
\underset{n>0}{\hocolim}\;\f_{-n}\f^{-1}(\E)
\overset{\simeq}{\longrightarrow}
\f^{-1}(\E).
\end{equation}

Next we make precise the vagary of identifying the homotopy fixed points $\f^{-1}(\KGL)^{hC_{2}}$ with a homotopy colimit.
That is, 
we identify a homotopy limit with a homotopy colimit.
Throughout we let $\E$ be a motivic spectrum equipped with a $G$-action for a finite group $G$.

\begin{lemma}
\label{lemma:hfphocolim}
There is a natural isomorphism
\begin{equation}
\label{equation:positivehfpmap}
\underset{n>0}{\hocolim}\;\f_{-n}\f^{-1}(\E)^{hG}
\overset{\simeq}{\longrightarrow}
\f^{-1}(\E)^{hG}.
\end{equation}
\end{lemma}
\begin{proof}
For every generator $\Sigma^{s,t}X_{+}$ of $\SH$ there is a canonically induced map
\begin{equation}
\label{equation:positivehfpmapgenerator}
\SH(\Sigma^{s,t}X_{+},\underset{n>0}{\hocolim}\;\f_{-n}\f^{-1}(\E)^{hG})
\longrightarrow
\SH(\Sigma^{s,t}X_{+},\f^{-1}(\E)^{hG}).
\end{equation}
If $t\geq 0$ the source and target of \eqref{equation:positivehfpmapgenerator} are trivial.
If $t<0$ we show \eqref{equation:positivehfpmapgenerator} is an isomorphism by using the distinguished triangle 
\begin{equation}
\label{equation:longeffectivecocoverdefinitionhfp}
\f_{t}\f^{-1}(\E)^{hG}
\longrightarrow 
\f^{-1}(\E)^{hG}
\longrightarrow 
\f^{t-1}(\E)^{hG}, 
\end{equation}
obtained by applying homotopy fixed points to \eqref{equation:effectivecocoverdefinition} for $\f^{-1}(\E)$ and identifying $\f^{t-1}\f^{-1}(\E)$ with $\f^{t-1}(\E)$ by means of the distinguished triangles
\begin{equation*}
\f_{t}(\e)
\overset{\simeq}{\longrightarrow}
\e
\longrightarrow 
\f^{t-1}(\e),
\;
\f^{t-1}(\e)
\longrightarrow 
\f^{t-1}(\E)
\overset{\simeq}{\longrightarrow}
\f^{t-1}\f^{-1}(\E).
\end{equation*}
Since $\f^{t-1}(\E)^{hG}$ in \eqref{equation:longeffectivecocoverdefinitionhfp} is $(t-1)$-coeffective there is a canonically induced isomorphism
\begin{equation*}
\SH(\Sigma^{s,t}X_{+},\f_{t}\f^{-1}(\E)^{hG})
\overset{\cong}{\longrightarrow}
\SH(\Sigma^{s,t}X_{+},\f^{-1}(\E)^{hG}).
\end{equation*}
On the other hand there are canonical identifications
\begin{align*}
\SH(\Sigma^{s,t}X_{+},\underset{n>0}{\hocolim}\;\f_{-n}\f^{-1}(\E)^{hG}) 
& \cong \underset{n>0}{\colim}\;\SH(\Sigma^{s,t}X_{+},\f_{-n}\f^{-1}(\E)^{hG}) \\
& \cong \SH(\Sigma^{s,t}X_{+},\f_{t}\f^{-1}(\E)^{hG}).
\end{align*}
\end{proof}

In the following we make the standing assumption that for all $q\in\Z$ there is a naturally induced isomorphism
\begin{equation}
\label{equation:hfpslicesmap}
\s_{q}(\E^{hG})
\overset{\simeq}{\longrightarrow}
\s_{q}(\E)^{hG}.
\end{equation}
The map in \eqref{equation:hfpslicesmap} arises from the standard adjunction between motivic spectra and ``naive'' $G$-motivic spectra.
That is, 
with the trivial $G$-action on the homotopy fixed points there is a naturally induced $G$-map $\s_{q}(\E^{hG})\rightarrow\s_{q}(\E)$. 
Its adjoint is the map in \eqref{equation:hfpslicesmap}.

\begin{corollary}
\label{corollary:longeffectivecovershfp}
Assuming \eqref{equation:hfpslicesmap} and $n>0$ there is a naturally induced isomorphism
\begin{equation}
\label{equation:longeffectivecovershfp}
\f_{-n}\f^{-1}(\E^{hG})
\overset{\simeq}{\longrightarrow}
\f_{-n}\f^{-1}(\E)^{hG}.
\end{equation}
\end{corollary}
\begin{proof}
Follows from Lemma \ref{lemma:negativedt} under the stated assumptions.
\end{proof}

\begin{corollary}
\label{corollary:negativesliceshfp}
Assuming \eqref{equation:hfpslicesmap} and $n\geq -q>0$ there is a naturally induced isomorphism
\begin{equation}
\label{equation:positivehfpmapslices}
\s_{q}(\E^{hG})
\overset{\simeq}{\longrightarrow}
\s_{q}(\underset{n>0}{\hocolim}\;\f_{-n}\f^{-1}(\E)^{hG}).
\end{equation}
\end{corollary}
\begin{proof}
From the isomorphisms \eqref{equation:negativesliceisomorphism} in Lemma \ref{lemma:negativesliceisomorphism} we obtain 
\begin{equation}
\label{equation:negativesliceisomorphismhfp}
\s_{q}(\E^{hG})
\overset{\simeq}{\longrightarrow}
\s_{q}\f^{-1}(\E^{hG})
\overset{\simeq}{\longleftarrow}
\s_{q}\f_{-n}\f^{-1}(\E^{hG}).
\end{equation}
Recall that slices commute with homotopy colimits \cite[Corollary 4.5]{Spitzweck}, \cite[Lemma 4.2]{voevodsky.open}.
Thus the target in \eqref{equation:positivehfpmapslices} identifies with the homotopy colimit 
\begin{equation}
\label{equation:slicecommutation}
\underset{n>0}{\hocolim}\;\s_{q}\f_{-n}\f^{-1}(\E)^{hG}.
\end{equation}
With the assumption $n\geq -q>0$ the $q$th slice $\s_{q}\f_{-n}\f^{-1}(\E)^{hG}$ maps isomorphically to \eqref{equation:slicecommutation}.
It remains to apply the isomorphism \eqref{equation:longeffectivecovershfp} in Corollary \ref{corollary:longeffectivecovershfp} and \eqref{equation:negativesliceisomorphismhfp}.
\end{proof}


\begin{proposition}
\label{proposition:hfpslices}
Assuming \eqref{equation:hfpslicesmap} the slices of $\e$ commute with homotopy fixed points in the sense that there is a naturally induced isomorphism
\begin{equation}
\label{equation:hfpsliceseffectivemap}
\s_{q}(\e^{hG})
\overset{\simeq}{\longrightarrow}
\s_{q}(\e)^{hG}
\end{equation}
for every $q\in\Z$.
Moreover, 
$\e^{hG}$ is an effective motivic spectrum. 
\end{proposition}
\begin{proof}
Applying homotopy fixed points to \eqref{equation:effectivecocoverdefinition} yields the distinguished triangle
\begin{equation}
\label{equation:effectivecocoverdefinitionhfp}
\e^{hG}
\longrightarrow 
\E^{hG}
\longrightarrow 
\f^{-1}(\E)^{hG}.
\end{equation}
From \eqref{equation:effectivecocoverdefinitionhfp} we deduce the commutative diagram of distinguished triangles:
\begin{equation}
\label{equation:slicediagramhfp}
\xymatrix{
\s_{q}(\e^{hG}) \ar[r] \ar[d] &
\s_{q}(\E^{hG}) \ar[r] \ar[d] &
\s_{q}(\f^{-1}(\E)^{hG}) \ar[d] \\
\s_{q}(\e)^{hG} \ar[r] &
\s_{q}(\E)^{hG} \ar[r] &
\s_{q}(\f^{-1}(\E))^{hG} 
}
\end{equation}

When $q\geq 0$ it follows that $\s_{q}(\f^{-1}(\E))\simeq\s_{q}(\f^{-1}(\E)^{hG})\simeq\ast$ since homotopy limits preserve coeffective motivic spectra.
Since the middle vertical map in \eqref{equation:slicediagramhfp} is an isomorphism, 
see the assumption \eqref{equation:hfpslicesmap}, 
so is \eqref{equation:hfpsliceseffectivemap}.

When $q<0$, 
\eqref{equation:positivehfpmap} and \eqref{equation:positivehfpmapslices} imply that $\s_{q}(\E^{hG})\longrightarrow\s_{q}(\f^{-1}(\E)^{hG})$ is an isomorphism.
Lemma \ref{lemma:nonegativeslices} implies $\e^{hG}\in\SH^{\eff}$ and thus $\s_{q}(\e^{hG})\longrightarrow\s_{q}(\e)^{hG}$ is an isomorphism.
\end{proof}

\begin{lemma}
\label{lemma:hfpcommute}
Assuming \eqref{equation:hfpslicesmap} there are naturally induced isomorphisms
\begin{equation}
\label{equation:hfpcommuteeffective}
\f^{q}(\e^{hG})
\overset{\simeq}{\longrightarrow}
\f^{q}(\e)^{hG}, 
\;
\f_{q}(\e^{hG})
\overset{\simeq}{\longrightarrow}
\f_{q}(\e)^{hG}.
\end{equation}
\end{lemma}
\begin{proof}
We show that all the nonnegative effective cocovers of $\e$ commute with homotopy fixed points.
With this in hand the assertion for the effective covers of $\e$ follows from \eqref{equation:effectivecocoverdefinition}.

We claim there is a commutative diagram:
\begin{equation}
\label{equation:degreezero}
\xymatrix{
\s_{0}(\e^{hG}) \ar[r]^-{\simeq} \ar[d]_-{\simeq} &
\f^{0}(\e^{hG}) \ar[d] \\
\s_{0}(\e)^{hG} \ar[r]^-{\simeq} &
\f^{0}(\e)^{hG}
}
\end{equation}
Proposition \ref{proposition:hfpslices} shows the left vertical map in \eqref{equation:degreezero} is an isomorphism and that $\e^{hG}$ is an effective motivic spectrum.
Hence $\s_{0}(\e)\simeq\f^{0}(\e)$ and $\s_{0}(\e^{hG})\simeq\f^{0}(\e^{hG})$ by comparing \eqref{equation:slicedefinition} and \eqref{equation:effectivecocoverdefinition}.
It follows that the natural map $\f^{0}(\e^{hG})\longrightarrow\f^{0}(\e)^{hG}$ is also an isomorphism.

The cone of the left vertical map in (\ref{equation:effectivediagram1}) is the $q$th slice.
Hence there is a homotopy cofiber sequence $\s_{q}(\e)\longrightarrow\f^{q}(\e)\longrightarrow\f^{q-1}(\e)$, 
and likewise for $\e^{hG}$.
Proposition \ref{proposition:hfpslices} and induction on $q$ implies that $\f^{q}(\e)$ commutes with homotopy fixed points.
\end{proof}

\begin{corollary}
\label{corollary:effectiveslicecomlete}
Assuming \eqref{equation:hfpslicesmap} there is a naturally induced isomorphism
\begin{equation}
\label{equation:hfpslicecomplete}
\scc(\e^{hG})
\overset{\simeq}{\longrightarrow}
\scc(\e)^{hG}.
\end{equation}
If $\e$ is slice complete then so is $\e^{hG}$.
\end{corollary}
\begin{proof}
Lemma \ref{lemma:hfpcommute} and the fact that homotopy limits commute imply there are canonical isomorphisms
\begin{equation*}
\underset{q\rightarrow\infty}{\holim}\;\f^{q-1}(\underset{G}{\holim}\;\e)
\simeq
\underset{q\rightarrow\infty}{\holim}\;\underset{G}{\holim}\;\f^{q-1}(\e)
\simeq
\underset{q\rightarrow\infty}{\holim}\;\underset{G}{\holim}\;\f^{q-1}(\e)
\simeq
\underset{G}{\holim}\;\underset{q\rightarrow\infty}{\holim}\;\f^{q-1}(\e).
\end{equation*}
For slice completeness of $\e^{hG}$ we use the factorization $\e^{hG}\longrightarrow\scc(\e^{hG})\longrightarrow\scc(\e)^{hG}$.
\end{proof}

\begin{corollary}
\label{corollary:hfpeffectivecovers}
Assuming \eqref{equation:hfpslicesmap} there are naturally induced isomorphisms
\begin{equation}
\label{equation:hfpcommute}
\f^{q}(\E^{hG})
\overset{\simeq}{\longrightarrow}
\f^{q}(\E)^{hG}, 
\;
\f_{q}(\E^{hG})
\overset{\simeq}{\longrightarrow}
\f_{q}(\E)^{hG}.
\end{equation}
\end{corollary}
\begin{proof}
There is a naturally induced commutative diagram of distinguished triangles:
\begin{equation}
\label{equation:slicediagram}
\xymatrix{
\f_{0}(\E^{hG}) \ar[r] \ar[d] &
\E^{hG} \ar[r] \ar@{=}[d] &
\f^{-1}(\E^{hG}) \ar[d] \\
\f_{0}(\E)^{hG} \ar[r] &
\E^{hG} \ar[r] &
\f^{-1}(\E)^{hG}
}
\end{equation}
Proposition \ref{proposition:hfpslices} shows that $\f_{0}(\E^{hG})\longrightarrow\e^{hG}$ is a map between effective motivic spectra. 
It follows that $\f^{-1}(\E^{hG})\longrightarrow\f^{-1}(\E)^{hG}$ induces an isomorphism on all negative slices.
Since it is a map between coeffective spectra, 
it is in fact an isomorphism according to Lemma \ref{lemma:sliceiso}.

The general cases follow by using induction on \eqref{equation:slicedefinition} and \eqref{equation:effectivecocoverdefinition}.
\end{proof}

We end this section by discussing $G$-fixed points in more detail. 
Let $\mathcal{O}_{G}$ denote the orbit category of $G$ with objects $\{G/H\}$ and morphisms the $G$-maps $\map_{G}(G/H,G/K)\cong (G/K)^{H}$
\cite[\S1.8]{MR2456045}.
Let $\MSS$ be a highly structured model for the stable motivic homotopy category, 
e.g., motivic functors \cite{MR2029171}, or motivic symmetric spectra \cite{jardine.symmetric}. 
Let $\MSS^{\eff}$ be the Bousfield colocalization of $\MSS$ with respect to the set of objects $\Sigma^{p,0}\Sigma^{\infty}_{\mathbf{P}^{1}}X_{+}$, 
where $X\in\Sm_{F}$ and $p\in\Z$ (it suffices to consider $p\leq 0$).
Its homotopy category is $\SH^{\eff}$.
As a model for naive $G$-motivic spectra we use the functor category $[\mathcal{O}_{G}^{\op},\MSS]$ with the projective model structure \cite[Theorem 11.6.1]{Hirschhorn}.
There is a naturally induced Quillen adjunction:
\begin{equation}
\label{equation:Gadjunction}
\xymatrix{
\ii_{0}^{G}
\colon
[\mathcal{O}_{G}^{\op},\MSS^{\eff}]
\ar@<0.5ex>[r]  
& 
\ar@<0.5ex>[l] 
[\mathcal{O}_{G}^{\op},\MSS]
\colon 
\rr_{0}^{G}
} 
\end{equation}

Evaluating a naive $G$-motivic spectrum $\underline{\E}$ at the orbits corresponding to $G$ and the identity element yields the underlying motivic spectrum $\E$ 
and the $G$-fixed points $\E^{G}$,
respectively.
Let $\underline{\e}$ be the naive $G$-motivic spectrum $\f_{0}^{G}(\underline{\E})$, 
where $\f_{0}^{G}=\mathbf{L}\ii_{0}^{G}\circ\rr_{0}^{G}$.
(Forgetting the $G$-action, $\underline{\e}$ coincides with $\e$.)
Since evaluating at the identity orbit commutes with \eqref{equation:Gadjunction}, 
we obtain an isomorphism
\begin{equation}
\label{equation:Gisos}
\f_{0}^{G}(\underline{\E})^{G}
\cong
\f_{0}(\underline{\E}^{G}).
\end{equation}
In particular, 
$\kgl^{C_{2}}$ coincides with the effective hermitian $K$-theory spectrum $\kq$, 
cf.~\eqref{equation:kohlp2}.

\section{The slice spectral sequence}
\label{section:sss}

The trigraded slice spectral sequence for $\E$ arising from \eqref{equation:slicefiltration} takes the form
\begin{equation}
\label{equation:splicespectralsequence2}
\pi_{\star}\s_{\ast}(\E)
\Longrightarrow
\pi_{\star}\E.
\end{equation}
This is an upper half-plane spectral sequence with entering differentials \cite[\S7]{Boardman} because $\pi_{p,w}\s_{q}(\E)=0$ for $q<w$, 
cf.~\cite[\S7]{voevodsky.open}.
A standard argument shows that (\ref{equation:splicespectralsequence2}) converges conditionally to the motivic homotopy groups of $\scc(\E)$ in the sense of \cite[Definition 5.10]{Boardman}.
For the following result we refer to \cite[Lemma 3.14]{roendigs-spitzweck-oestvaer.slices-sphere}.
\begin{lemma}
\label{lemma3}
Suppose $\e\in\SH^{\eff}$ and $\e/\eta$ is slice complete. 
Then there is a naturally induced isomorphism between $\e^{\wedge}_{\eta}$ to $\scc(\e)$.
\end{lemma}

\begin{proposition}
\label{proposition:ccsss}
Suppose $\e\in\SH^{\eff}$ and $\e/\eta$ is slice complete. 
There is a conditionally convergent slice spectral sequence
\begin{equation}
\label{equation:splicespectralsequence3}
\pi_{\star}\s_{\ast}(\e)
\Longrightarrow
\pi_{\star}\e^{\wedge}_{\eta}.
\end{equation}
\end{proposition}
\begin{proof}
This follows from Lemma \ref{lemma3} and \eqref{equation:splicespectralsequence2}.
\end{proof}

\section{Proofs of Theorems \ref{theorem:kqetahomotopylimit} and \ref{theorem:KQetahomotopylimit}}
\label{section:proofsofmainresults}

\begin{corollary}
\label{corollary:ccsssforkgl}
For effective hermitian $K$-theory there is a conditionally convergent slice spectral sequence
\begin{equation}
\label{equation:splicespectralsequencekq}
\pi_{\star}\s_{\ast}(\kq)
\Longrightarrow
\pi_{\star}\kq^{\wedge}_{\eta}.
\end{equation}
\end{corollary}
\begin{proof}
The only issue is to identify the quotient of $\kq$ by $\eta$ with a slice complete spectrum.
By \cite[Theorem 3.4]{R-O} there is a homotopy cofiber sequence 
\begin{equation}
\label{equation:etacofiberkqkgl}
\Sigma^{1,1}\KQ
\overset{\eta}{\longrightarrow} 
\KQ
\longrightarrow
\KGL
\end{equation}
relating algebraic and hermitian $K$-theory via $\eta$.
Passing to effective covers in \eqref{equation:etacofiberkqkgl} identifies the cofiber of $\f_0(\eta\colon\Sigma^{1,1}\KQ\longrightarrow\KQ)$ with $\kgl$. 
Hence the cofiber of $\eta\colon\Sigma^{1,1}\kq\longrightarrow\kq$ is an extension of $\kgl$ by $\Sigma^{1,1}\s_{-1}(\KQ)\simeq\s_{0}(\Sigma^{1,1}\KQ)$, 
cf.~\cite[Lemma 2.1]{R-O}, 
so it is slice complete by Lemmas \ref{lemma:slicesareslicecomplete} and \ref{lemma:Ktheoryslicecomplete}. 
This verifies the assumptions in Proposition \ref{proposition:ccsss} for $\kq$.
\end{proof}

Assuming $\vcd_{2}(F)<\infty$, \eqref{equation:hfpslicesmap} holds for $\KGL$ by \cite[Proposition 4.24]{R-O}.
We summarize some useful consequences of the results in \S\ref{section:sf}.
\begin{proposition}
\label{proposition:summary}
The following holds when $\vcd_{2}(F)<\infty$.
\begin{itemize}
\item[(1)]
The homotopy fixed points spectrum of effective $K$-theory $\kgl^{hC_{2}}$ is slice complete.
There is a conditionally convergent slice spectral sequence
\begin{equation}
\label{equation:splicespectralsequencekgl}
\pi_{\star}\s_{\ast}(\kgl^{hC_{2}})
\Longrightarrow
\pi_{\star}\kgl^{hC_{2}}.
\end{equation}
\item[(2)]
There is a naturally induced isomorphism $\f_{q}(\KGL^{hC_{2}})\overset{\simeq}{\longrightarrow}\f_{q}(\KGL)^{hC_{2}}$.
\end{itemize}
\end{proposition}
\begin{proof}
Here (1) follows from Lemma \ref{lemma:Ktheoryslicecomplete}, Corollary \ref{corollary:effectiveslicecomlete} and the discussion of \eqref{equation:splicespectralsequence2} in \S\ref{section:sss}, 
while (2) is a special case of Corollary \ref{corollary:hfpeffectivecovers}.
\end{proof}

\begin{proof}[Proof of Theorem \ref{theorem:kqetahomotopylimit}]
By \cite[Theorems 4.18, 4.25, 4.27, Lemma 4.26]{R-O} the natural map $\Upsilon\colon\KQ\longrightarrow\KGL^{hC_{2}}$ in (\ref{equation:bigKQhlp}) induces an isomorphism of slices
\begin{equation}
\label{equation:KGLsliceisomorphism}
\s_{q}(\Upsilon)
\colon
\s_{q}(\KQ)
\overset{\simeq}{\longrightarrow}
\s_{q}(\KGL^{hC_{2}})
\simeq
\left\{
\begin{array}{ll}
\Sigma^{2q,q} \MZ \vee \bigvee_{i<0} \Sigma^{2q+2i,q} \mathsf{M}\mathbb{Z}/2 & q\equiv 0 (2), \\
\bigvee_{i<0} \Sigma^{2q+2i+1,q} \mathsf{M}\mathbb{Z}/2 & q\equiv 1 (2).
\end{array}
\right.
\end{equation}
From \eqref{equation:KGLsliceisomorphism} we conclude the natural map $\kq\longrightarrow\f_{0}(\KGL^{hC_{2}})$ induces an isomorphism on slices.
By composing with $\f_{0}(\KGL^{hC_{2}})\overset{\simeq}{\longrightarrow}\kgl^{hC_{2}}$, 
see Proposition \ref{proposition:summary}(2), 
we conclude there is an isomorphism
\begin{equation}
\label{equation:kglsliceisomorphism}
\s_{q}(\gamma)
\colon
\s_{q}(\kq)
\overset{\simeq}{\longrightarrow}
\s_{q}(\kgl^{hC_{2}}).
\end{equation}
Thus $\gamma\colon\kq\longrightarrow\kgl^{hC_{2}}$ in \eqref{equation:kohlp2} induces an isomorphism between the conditionally convergent upper half-plane slice spectral sequences 
\eqref{equation:splicespectralsequencekq} and \eqref{equation:splicespectralsequencekgl}.
The induced map between the filtered target groups is thus an isomorphism \cite[Theorem 7.2]{Boardman}.
This finishes the proof by passing to Nisnevich sheaves of homotopy groups.
\end{proof}

\begin{proof}[Proof of Theorem \ref{theorem:KQetahomotopylimit}]
Using \eqref{equation:effectivecoverdefinition} we obtain the naturally induced commutative diagram of distinguished triangles: 
\begin{equation}
\label{equation:diagram5}
\xymatrix{
\kq
\ar[r] \ar[d]_{\f_{0}(\Upsilon)} &
\KQ
\ar[r] \ar[d]_{\Upsilon} &
\f^{-1}(\KQ)
\ar[d]_{\f^{-1}(\Upsilon)} \\
\f_{0}(\KGL^{hC_{2}})
\ar[r] &
\KGL^{hC_{2}}
\ar[r] &
\f^{-1}(\KGL^{hC_{2}})
}
\end{equation}
Lemma \ref{lemma:sliceiso} and \eqref{equation:KGLsliceisomorphism} imply that $\f^{-1}(\Upsilon)\colon \f^{-1}(\KQ)\longrightarrow\f^{-1}(\KGL^{hC_{2}})$ is an isomorphism because it is a map 
between coeffective spectra and it induces on isomorphism on slices.
By composing with the isomorphism $\f_{0}(\KGL^{hC_{2}})\overset{\simeq}{\longrightarrow}\kgl^{hC_{2}}$ of Proposition \ref{proposition:summary}(2) we obtain a commutative diagram 
of distinguished triangles:
\begin{equation}
\label{equation:diagram6}
\xymatrix{
\kq
\ar[r] \ar[d]_{\gamma} &
\KQ
\ar[r] \ar[d]_{\Upsilon} &
\f^{-1}(\KQ)
\ar[d]_{\f^{-1}(\Upsilon)}^{\simeq} \\
\kgl^{hC_{2}}
\ar[r] &
\KGL^{hC_{2}}
\ar[r] &
\f^{-1}(\KGL^{hC_{2}})
}
\end{equation}
This shows that $\LL_{\eta}(\gamma)$ is an isomorphism if and only if $\LL_{\eta}(\Upsilon)$ is an isomorphism.
It follows that Theorem \ref{theorem:kqetahomotopylimit} implies Theorem \ref{theorem:KQetahomotopylimit}.
\end{proof}

\begin{corollary}
\label{corollary:KWtateK}
If $\vcd_{2}(F)<\infty$ then the $\eta$-completion of the Tate $K$-theory spectrum ${\KGL}^{tC_{2}}$ is contractible.
\end{corollary}
\begin{proof}
Recall that ${\KGL}^{tC_{2}}$ is the cone of the norm map from the homotopy orbits $\KGL_{hC_{2}}$ to the homotopy fixed points $\KGL^{hC_{2}}$ in the Tate diagram \cite[(20)]{HKO}
\begin{equation*}
\xymatrix{
\KGL_{hC_{2}} \ar[r] \ar@{=}[d] & \KQ \ar[r] \ar[d]_{\Upsilon} & \KW \ar[d] \\
\KGL_{hC_{2}} \ar[r] & \KGL^{hC_{2}}  \ar[r] & {\KGL}^{tC_{2}}
}
\end{equation*}
for the $C_{2}$-action on $\KGL$. 
Thus the assertion follows from Theorem \ref{theorem:KQetahomotopylimit} since the higher Witt-theory spectrum $\KW$ can be identified with $\KQ[\eta^{-1}]$,
see e.g., \cite[(7)]{R-O}. 
\end{proof}

\section{Appendix}
\label{section:appendix}
By way of example we show that $\SH^{\eff}$ is not closed under homotopy fixed points.
In effect, 
consider the homotopy fixed points for the trivial $C_{2}$-action on the effective motivic spectrum 
\begin{equation}
\bigvee_{i\geq 0} \Sigma^{i,0}\MZ/2\simeq
\prod_{i\geq 0} \Sigma^{i,0}\MZ/2.
\end{equation} 
Here the sum and product are isomorphic by \cite[Proposition A.5]{R-O}.
Since $C_{2}$-homotopy fixed points commute with products, \cite[Lemma 4.22]{R-O} yields a naturally induced isomorphism
\begin{equation}
\label{equation:MZhfpiso}
\Bigl(\prod_{i\geq 0} \Sigma^{i,0}\MZ/2\Bigr)^{hC_{2}}
\overset{\simeq}{\longrightarrow}
\prod_{i\geq 0} \prod_{j\geq 0} \Sigma^{i-j,0}\MZ/2.
\end{equation}
Assuming \eqref{equation:MZhfpiso} is an isomorphism in $\SH^{\eff}$, 
the countably infinite product $\prod_{n\in \mathbb{N}}\MZ/2$ --- corresponding to indices $i=j$ --- is effective.
However, 
we show that $\prod_{n\in \mathbb{N}}\MZ/2\not\in\SH^{\eff}$.

Recall from \cite[\S2]{Voevodsky:motivicss} the adjunction between $\SH$ and the stable motivic homotopy category of $S^{1}$-spectra $\SH_{s}$: 
\begin{equation*}
\xymatrix{
\Sigma^{\infty}_{t} \colon \SH_{s} \ar@<0.5ex>[r] &
\SH \colon \Omega^{\infty}_{t}. \ar@<0.5ex>[l]
}
\end{equation*}
Now $\Omega^{\infty}_{t}\MZ/2$ is the Eilenberg-MacLane $S^1$-spectrum $\HZ/2$ associated with the constant presheaf $\mathbf{Z}/2$ by 
\cite[Lemma 5.2]{Voevodsky:motivicss}.
It follows that $\Omega^{\infty}_{t}\prod_{n\in\mathbb{N}}\MZ/2$ is the Eilenberg-MacLane $S^1$-spectrum $\HV$ associated with the constant presheaf $\mathbf{V}$, 
where $\mathbf{V}$ is a $\mathbf{Z}/2$-vector space of (uncountable) infinite dimension.
If $\prod_{n\in \mathbb{N}}\MZ/2\in\SH^{\eff}$ we would obtain 
\begin{equation}
\label{equation:1effectiveassumption}
\Omega^{\infty}_{t}\Sigma^{0,1}\prod\MZ/2
\simeq
\prod\Omega^{\infty}_{t}\Sigma^{0,1}\MZ/2
\in\Sigma^{1}_{t}\SH_ {s}
\end{equation}
since $\Omega^{\infty}_{t}\Sigma^{0,1}\MZ/2\in\Sigma^{1}_{t}\SH_ {s}$ by \cite[Theorem 7.4.1]{Levine:slices} (as conjectured in \cite[Conjecture 4]{Voevodsky:motivicss}).
In~(\ref{equation:1effectiveassumption}) we use that $\Sigma^{0,1}$, being an equivalence, commutes with products.
In $\SH_{s}$ there is a canonically induced map 
\begin{equation}
\label{equation:canmap}
\alpha
\colon
(\prod_{n\in \mathbb{N}} \Omega^{\infty}_{t}\MZ/2)
\otimes_{\Omega^{\infty}_{t}\MZ/2} 
\Omega^{\infty}_{t}\Sigma^{0,1}\MZ/2
\longrightarrow
\prod_{n\in \mathbb{N}} \Omega^{\infty}_{t}\Sigma^{0,1}\MZ/2.
\end{equation}
The tensor product in \eqref{equation:canmap} is formed in the module category of the Eilenberg-MacLane $S^1$-spectrum $\Omega^{\infty}_{t}\MZ/2$ \cite{MR3309296}, 
cf.~\cite{Rondigs-Ostvar1}, \cite{Rondigs-Ostvar2}, \cite[Lemma 5.2]{Voevodsky:motivicss}. 
In particular, 
the source of $\alpha$ in~(\ref{equation:canmap}) is $1$-effective, 
because $\Omega^{\infty}_{t}\Sigma^{0,1}\MZ/2\in\Sigma^{1}_{t}\SH_ {s}$ by \cite[Theorem 7.4.1]{Levine:slices} and $\prod_{n\in \mathbb{N}} \Omega^{\infty}_{t}\MZ/2$ is effective,
since it is the Eilenberg-MacLane spectrum associated with $\mathbf{V}$.
We will prove that $\SH_{s}(\Sigma^{n,0}X_{+}\wedge\G,\alpha)$ is an isomorphism for every $X\in\Sm_{F}$ and $n\in\Z$;  
here $\G$ denotes the multiplicative group scheme.
In effect, 
choose an uncountable basis $\mathcal{B}$ of $\mathbf{V}$ and express $\mathbf{V}$ as the filtered colimit of finite dimensional sub-$\mathbf{Z}/2$-vector spaces 
$\mathbf{V}^\prime \subset \mathbf{V}$ spanned by finite subsets $\mathcal{F}\subset \mathcal{B}$. 
Since $\Sigma^{n,0}X_{+}\wedge \G$ is compact in $\SH_{s}$ there are isomorphisms
\begin{align*}
& 
\SH_ {s}\Bigl(\Sigma^{n,0}X_{+}\wedge\G,\bigl(\prod_{n\in \mathbb{N}} \Omega^{\infty}_{t}\MZ/2\bigr)\otimes_{\Omega^{\infty}_{t}\MZ/2}\Omega^{\infty}_{t}\Sigma^{0,1}\MZ/2\Bigr) \cong \\
&
\SH_ {s}\Bigl(\Sigma^{n,0}X_{+}\wedge\G,(\colim_{\mathcal{F}\subset\mathcal{B}} \prod_{f\in \mathcal{F}}\Omega^{\infty}_{t}\MZ/2)\otimes_{\Omega^{\infty}_{t}\MZ/2}\Omega^{\infty}_{t}\Sigma^{0,1}\MZ/2\Bigr) \cong \\
&
\colim_{\mathcal{F}\subset\mathcal{B}} \;
\SH_ {s}\Bigl(\Sigma^{n,0}X_{+}\wedge\G,(\prod_{f\in \mathcal{F}}\Omega^{\infty}_{t}\MZ/2\otimes_{\Omega^{\infty}_{t}\MZ/2}\Omega^{\infty}_{t}\Sigma^{0,1}\MZ/2)\Bigr) \cong \\
&
\colim_{\mathcal{F}\subset\mathcal{B}} \;
\SH_ {s}\Bigl(\Sigma^{n,0}X_{+}\wedge\G,\Omega^{\infty}_{t}\MZ/2 \otimes_{\Omega^{\infty}_{t}\MZ/2} 
\prod_{f\in \mathcal{F}}\Omega^{\infty}_{t}\Sigma^{0,1}\MZ/2\Bigr)   \cong \\
&
\colim_{\mathcal{F}\subset\mathcal{B}} \;
\SH_ {s}\Bigl(\Sigma^{n,0}X_{+}\wedge \G,\prod_{f\in \mathcal{F}} \Omega^{\infty}_{t}\Sigma^{0,1}\MZ/2\Bigr) \cong \\
&
\colim_{\mathcal{F}\subset\mathcal{B}} \;
\SH_ {s}\Bigl(\Sigma^{n,0}X_{+},\Omega_{t}\bigl(\prod_{f\in \mathcal{F}} \Omega^{\infty}_{t}\Sigma^{0,1}\MZ/2\bigr)\Bigr) \cong \\
&
\colim_{\mathcal{F}\subset\mathcal{B}} \;
\SH_ {s}\Bigl(\Sigma^{n,0}X_{+},\prod_{f\in \mathcal{F}} \Omega^{\infty}_{t}\Omega_{t}\Sigma^{0,1}\MZ/2\Bigr) \cong \\
&
\colim_{\mathcal{F}\subset\mathcal{B}} \;
\SH_ {s}\Bigl(\Sigma^{n,0}X_{+},\prod_{f\in \mathcal{F}} \Omega^{\infty}_{t}\Sigma^{-1,0}\MZ/2\Bigr) \cong \\
&
\SH_ {s}\Bigl(\Sigma^{n,0}X_{+},\colim_{\mathcal{F}\subset\mathcal{B}} \;
\prod_{f\in \mathcal{F}} \Omega^{\infty}_{t}\Sigma^{-1,0}\MZ/2\Bigr) \cong \\
&
\SH_ {s}\Bigl((\Sigma^{n,0}X_{+},\prod_{n\in \mathbb{N}} \Omega^{\infty}_{t}\Sigma^{-1,0}\MZ/2\Bigr ) \cong 
\SH_ {s}\Bigl(\Sigma^{n,0}X_{+},\prod_{n\in \mathbb{N}} \Omega^{\infty}_{t}\Omega_{t}\Sigma^{0,1}\MZ/2\Bigr) \cong \\
&
\SH_ {s}\Bigl(\Sigma^{n,0}X_{+},\Omega_{t}\prod_{n\in \mathbb{N}} \Omega^{\infty}_{t}\Sigma^{0,1}\MZ/2\Bigr) \cong
\SH_ {s}\Bigl(\Sigma^{n,0}X_{+}\wedge \G,\prod_{n\in \mathbb{N}} \Omega^{\infty}_{t}\Sigma^{0,1}\MZ/2\Bigr),
\end{align*}
which by canonicity coincides with the map induced by $\alpha$.
If the target in \eqref{equation:canmap} is $1$-effective, 
as implied by \eqref{equation:1effectiveassumption}, 
it would follow that $\alpha$ is an isomorphism.
One checks that $\alpha$ is not an isomorphism by choosing a field $F$ such that $F^{\ast}\otimes \mathbb{Z}/2$ is an infinitely generated $\mathbb{Z}/2$-module, 
e.g., 
$F=\mathbb{Q}$.
The map $\SH_{s}(\Sigma^{-1,0}\mathrm{Spec}(F)_+,\alpha)$ coincides with the canonical map 
\begin{equation}
\Bigl(\prod \mathbb{Z}/2\Bigr) \otimes_{\mathbb{Z}/2}  (F^{\ast}\otimes \mathbb{Z}/2)
\longrightarrow
\prod F^{\ast}\otimes \mathbb{Z}/2,
\end{equation}
which is not surjective.
Hence $\prod_{n\in \mathbb{N}} \Omega^{\infty}_{t}\Sigma^{0,1}\MZ/2$ cannot be $1$-effective.
As explained above it follows that $(\prod_{i\geq 0} \Sigma^{0,i}\MZ/2)^{hC_{2}}$ is noneffective.

{\bf Acknowledgments.}
The authors gratefully acknowledge support from DFG GK 1916, DFG Priority Program 1786, and RCN Frontier Research Group Project no.~250399.

\begin{footnotesize}

\end{footnotesize}
\vspace{0.1in}

\begin{small}
\begin{center}
Mathematisches Institut, Universit\"at Osnabr\"uck, Germany.\\
e-mail: oliver.roendigs@uni-osnabrueck.de
\end{center}
\begin{center}
Mathematisches Institut, Universit\"at Osnabr\"uck, Germany.\\
e-mail: markus.spitzweck@uni-osnabrueck.de
\end{center}
\begin{center}
Department of Mathematics, University of Oslo, Norway.\\
e-mail: paularne@math.uio.no
\end{center}
\end{small}
\end{document}